\documentclass[11pt]{amsart}

\usepackage{amssymb}
\usepackage{epstopdf}
\usepackage{amsmath}
\usepackage{amscd}
\usepackage{amsthm,amsfonts,amsxtra,amssymb,amscd}
\usepackage[all]{xy}
\usepackage{enumerate}

\usepackage{hyperref}

\theoremstyle{plain}

\usepackage{hyperref}

\theoremstyle{plain}
\newtheorem*{lemma*}{Lemma}
\newtheorem{lemma}[subsection]{Lemma}
\newtheorem*{theorem*}{Theorem}
\newtheorem{theorem}[subsection]{Theorem}
\newtheorem*{proposition*}{Proposition}

\newtheorem*{corollary*}{Corollary}

\theoremstyle{definition}
\newtheorem*{definition*}{Definition}
\newtheorem{definition}[subsection]{Definition}
\newtheorem*{example*}{Example}

                                    \theoremstyle{remark}
\newtheorem*{remark*}{Remark}

\newcommand{\R}{{\mathbb R}}
\newcommand{\deps}{\frac{d}{d\epsilon}|_{0}}

\begin{document}

\title{ A homotopy  for a complex of free Lie algebras}

 \author{ Mich{\`e}le Vergne}
\date{}
\maketitle

\address{Mich{\`e}le Vergne, Institut de Math\'ematiques de Jussieu, Th{\'e}orie des
Groupes, Case 7012, 2 Place Jussieu, 75251 Paris Cedex 05, France}

\address{email: vergne@math.jussieu.fr}

%AMS Classification Number: 16 S
%
%Keywords: Free  Lie Algebras.

\section*{Abstract}
Using the Guichardet construction, we compute the cohomology groups of a complex of free Lie algebras introduced  by Alekseev and Torossian.

\section*{Introduction}

In their study of the relation between the KV-conjecture and  Drinfeld's associators, Alekseev and Torossian \cite{AT} studied  the Eilenberg-MacLane   differential $\delta_A:L_n\to L_{n+1}$  where $L_n$ is the free Lie algebra in $n$ variables, and computed the cohomology groups of $\delta_A$ in dimensions $1,2$. Following the construction of Guichardet \cite{gui1} (see also \cite{gui2}), we remark that the complex $\delta_A$ is acyclic, except in dimensions $1,2$, where the cohomology is of dimension $1$.
We also identify the cohomology groups of a similar complex
 $\delta_A: T_n\to T_{n+1}$ where $T_n$ is the free associative algebra in $n$ variables: the cohomology is of dimension $1$ in any degree. The Guichardet construction provides an explicit homotopy.

 Alekseev and Torossian used the computations in dimension $2$  to deduce the existence of a solution to the KV problem from the existence of an associator.  A simple by-product of their computation is the existence and the uniqueness  of the Campbell-Hausdorff formula.
 We do not have any other application of the computations of higher cohomologies.

In this note, we start with a review  of the construction of Guichardet. Then we adapt it  to free associative algebras and free Lie algebras.

I am thankful to the referee for his careful reading.

\section{The Guichardet  construction.}
Let $V$ be a finite dimensional real vector space.
Let $F^n$ be the space of polynomial functions
$f$ on  $V\oplus V\oplus\cdots\oplus  V$.
An element $f$ of $F^n$ is written as $f(v_1,v_2,\ldots, v_n)$.

Define
$$(\delta_nf)(v_1, \ldots , v_{n+1})=\sum_{i=1}^n(-1)^{i}f(v_1, v_2, \ldots, v_{i-1},{\hat v_i}, v_{i+1},\ldots, v_n).$$

For example:
$$(\delta_1 f)(v_1,v_2)=-f(v_2)+f(v_1)$$
$$(\delta_2 f)(v_1,v_2,v_3)=-f(v_2,v_3)+f(v_1,v_3)-f(v_1,v_2).$$

\bigskip

We define $F^0=\R$, and embed $F^0\to F^1$ as the constant functions.

The complex $0\to F^0\to F^1\to \cdots$ is acyclic except in degree $0$. Indeed $s:F^n\to F^{n-1}$  given by

\begin{equation}\label{ret}
(sf)(v_1,v_2,\ldots, v_{n-1})=f(0,v_1,v_2,\ldots, v_{n-1})
\end{equation} satisfies
${\rm Id}:=s\delta+\delta s$.

Now the additive group $V$ operates on $F^n$ by translations:
if $\alpha\in V$, we write $$(\tau(\alpha)f)(v_1,\ldots, v_n)=f(v_1-\alpha,\ldots, v_n-\alpha).$$
The differential $\delta$ commutes with  translations, so that it induces a differential $\delta_A$ on  the subspace of  translation invariant functions.

It is well known that
the cohomology of the complex  $\delta_A$ is isomorphic  with $\Lambda^{n-1}V^*$.
Here we recall Guichardet's explicit construction of the isomorphism
as we will adapt it to the "universal case" considered by Alekseev-Torossian.

 \bigskip

Let $\Omega^{n-1}$  be the space of  differential forms of exterior degree  $n-1$ on $V$, with polynomial coefficients, equipped  with the de Rham differential.

Consider the simplex $S:=S_{v_1,v_2,\ldots, v_n}$ in $V$ with vertices $(v_1,v_2,\ldots, v_n)$.
Thus the map $\Omega^{n-1} \to F^n$ defined by $\omega\to \int_{S} \omega$ induces a map from $\Omega^{n-1}$ to $F^n$. This map commutes with the differentials (as follows from Stokes formula) and with the natural action by translations.

     Conversely,  associate to  $f\in F^n$ a differential form $\omega(f)$  of degree $(n-1)$ by setting for $v_1,v_2,\ldots, v_{n-1}$ vectors in $V$, identified with tangent vectors at $v\in V$:
    $$\langle \omega(f)(v), v_1\wedge v_2\wedge \cdots\wedge v_{n-1}\rangle=
    \sum_{\sigma\in \Sigma_{n-1}}\epsilon(\sigma)\deps f(v, v+\epsilon_1 v_{\sigma(1)},\ldots, v+\epsilon_{n-1} v_{\sigma(n-1)}).$$

Here if $\phi$ is a  polynomial function of $\epsilon_1,\ldots,\epsilon_{n-1}$,  we employ the notation $\deps \phi(\epsilon)$ for
the coefficient of $\epsilon_1\cdots\epsilon_{n-1}$
in $\phi$.

    The map $\omega$ commutes with the differential, and  with the action of $V$ by translations.
    Thus the map $P_n:F^n\to F^n$ defined by
    $$P_n(f)=\int_{S}\omega(f)$$ produces a map from $F^n\to F^n$, commuting with the action of $V$.
    This map is the identity on $F^1$.

    Let us give the formulae for $P_n$ so that we see that the map $P_n$ is ``universal".

Given $v:=(v_1,v_2,\ldots, v_n)\in V$,
 consider the map $p_v: \R^{n-1}\to V$ given by
 $$p_v(t_1,t_2,\ldots, t_{n-1})=v_1+t_1(v_2-v_1)+\cdots+t_{n-1}(v_n-v_1).$$

  This  map sends the standard simplex $\Delta_{n-1}$ defined by
  $$t_i\geq 0, \sum_{i=1}^{n-1}t_i\leq 1$$
  to the simplex $S$ in $V$ with vertices $v_1,v_2,\ldots, v_n$.

Let us consider the form $$p_v^*\omega(f)=f(t,v)dt_1\wedge \cdots \wedge dt_{n-1}.$$

The map
$P_n$ is given by
$$(P_nf)(v)=\int_{\Delta_{n-1}} f(t,v)dt$$
where $f(t,v)$ is the element of $F_n$ depending on $t$ described as follows.

\begin{lemma}

Let $$v(t)=v_1+t_1(v_2-v_1)+\cdots+t_{n-1} (v_{n}-v_1).$$
Define

\begin{equation}\label{eqf}
f(t,v_1,v_2,\ldots, v_{n})=\deps \sum_{\sigma\in \Sigma[2,\ldots, n]}\epsilon(\sigma) f(v(t),v(t)+\epsilon_1 (v_{\sigma(2)}-v_1),\ldots, v(t)+\epsilon_{n-1}(v_{\sigma(n)}-v_1)).
\end{equation}

Here $t=(t_1,t_2,\ldots,t_{n-1})$ and
$\Sigma([2,\ldots, n])$ is the group of permutations of the set with $(n-1)$ elements  $[2,\ldots, n]$.

Then we have the formula

$$(P_nf)(v_1,v_2,\ldots,v_n)=\int_{\Delta_{n-1}} f(t,v)dt_1dt_2\cdots dt_{n-1}.$$

\end{lemma}

   Let $H:={\rm Id}-P$.
    Using the injectivity of the vector spaces $F^n$ in the category of  $V$-modules,  it is standard, and we will review the procedure below,  to  produce a homotopy
    $$G: F^{n}\to F^{n-1}$$ commuting with the action of $V$ by translations and such that:$$H=G\delta+\delta G.$$

We first use the following injectivity lemma.
\begin{lemma}\label{injective}
Let $A,B$ be two real vector spaces provided with a structure of $V$-modules.
Let $u: A\to F^n$  be a  $V$-module map from $A$ to $F^n$. Let $v: A\to B$ be an injective map of $V$-modules,
Then there exists a map $w: B\to F^n$ of $V$-modules
 extending $u$.
\end{lemma}

The formula for a map $w$ (depending on a choice of retraction) is given below in the proof:

\begin{proof}
Denote by $\tau$ the action of $V$ on $B$.
Let $s$ be a linear map from $B$ to $A$ such that $sv={\rm Id}$.
Let $b\in B$: we define the map $w$ (depending on our choice of linear retraction $s$) by
$$w(b)(v_1,v_2,\ldots,v_n)=u(s\tau(-v_1)b)(0,v_2-v_1,\ldots, v_{n}-v_1).$$
We verify that $b$ satisfy the wanted conditions.
 The crucial point is that  the map $w$ is a map of $V$-modules, as we now show.  Indeed
 $$w(\tau(v_0)b)(v_1,v_2,\ldots,v_n)=u(s(\tau(-v_1)\tau(v_0)b))(0,v_2-v_1,\ldots,v_n-v_1)$$
$$=
u(s(\tau(-v_1+v_0)b))(0,v_2-v_1,\ldots,v_n-v_1)$$
  while
$$(\tau(v_0)w(b)) (v_1,v_2,\ldots,v_n)=w(b)(v_1-v_0,v_2-v_0,\ldots, v_n-v_0)$$
  $$=u(s\tau(v_0-v_1)b)(0,v_2-v_1,\ldots, v_n-v_1).$$

\end{proof}

We now apply this lemma to define $G$ inductively.
  Consider the injective map deduced from $\delta$ from $F^{n}/\delta(F^{n-1})$ to $F^{n+1}$.

  Recall our linear map $s:F^{n+1}\to F^n$ given by Equation  (\ref{ret}).
   We may take as linear inverse (that we still call $s$) the map $s: F^{n+1}\to F^n$  followed by the projection $F^n\to F^{n}/\delta(F^{n-1})$.

We define $G^1=0$ and inductively $G^{n+1}$ as  the map extending
$$H^n-\delta G^n: F^{n}\to F^n$$ to $F^{n+1}$ constructed in Lemma \ref{injective}.  Indeed
$(H^n-\delta G^n)\delta=\delta H^{n-1}-\delta (-\delta G^{n-1}+H^{n-1})=0$ so that the map
$H^n-\delta G^n$ produces a map from $F^n/\delta (F^{n-1})\to F^{n}$ and we use the fact that
 $F^n/\delta (F^{n-1})$ is embedded in $F^{n+1}$  via $\delta$ with inverse $s$.

More precisely, given $v_1$ and $f\in F^{n+1}$, we define the function $\phi$ of $n$ variables given by
$$\phi(w_1,w_2,\ldots, w_n)=f(v_1,v_1+w_1,\ldots, v_1+w_n)$$

and define

$$(G^{n+1}f)(v_1,v_2,\ldots, v_n)=((H^n-\delta G^n)\phi)(0,v_2-v_1,\ldots, v_n-v_1).$$

For example, this leads to the following formulae  for the first elements $G^i$.

We have $G^1=0, G^2=0$.

 $$(G^3f)(v_1,v_2)=f(v_1,v_1,v_2)-\int_0^1 \deps f(v_1,v_1+t(v_2-v_1),v_1+t(v_2-v_1)+\epsilon(v_2-v_1))dt.$$

$$(G^4f)(v_1,v_2,v_3)= G^4_0+G^4_1+G^4_2$$
with
$$(G_0^4f)(v_1,v_2,v_3)=
f(v_1,v_1,v_2,v_3)-f(v_1,v_2,v_2,v_3)+f(v_1,v_1,v_1,v_3)-f(v_1,v_1,v_1,v_2),$$

$$(G_1^4f)(v_1,v_2,v_3)=\int_{t=0}^1 \deps f(v_1,v_2,v_2+t(v_3-v_2), v_2+(t+\epsilon)(v_3-v_2))$$
$$-\int_{t=0}^1 \deps f(v_1,v_1,v_1+t(v_3-v_1), v_1+(t+\epsilon)(v_3-v_1))$$
$$+
\int_{t=0}^1 \deps f(v_1,v_1,v_1+t(v_2-v_1), v_1+(t+\epsilon)(v_2-v_1)).
$$
$$(G_2^4f)(v_1,v_2,v_3)=-\int_{t\in S_2} \deps f(v_1,V(t),V(t+\epsilon_1), V(t+\epsilon_2))$$
$$-\int_{t\in \Delta_2} \deps
f(v_1,V(t),V(t+\epsilon_2), V(t+\epsilon_1)).$$
Here $t=[t_1,t_2]$,
$t+\epsilon_1=[t_1+\epsilon_1,t_2]$, $t+\epsilon_2=[t_1,t_2+\epsilon_2]$,
 $V(t)=v_1+t_1(v_2-v_1)+t_2(v_3-v_1)$,
 and $\Delta_2:=\{[t_1,t_2], t_1\geq 0, t_2\geq 0; t_1+t_2\leq 1\}$.

\bigskip

Let us now consider the action of $V$ by translations on the complex $F^n$.  The differential $\delta$ induces  a differential $\delta_A:F^n_A\to F^n_A$ on the subspaces of invariants.
We identify the space $F^{n}_A$ with  $F^{n-1}$ by the map
$$R: F^{n-1}\to F^n_A$$ given by

$$(Rf)(v_1,v_2,\ldots, v_n)=f(v_2-v_1,v_3-v_2,\ldots, v_n-v_{n-1}).$$

Then the differential $\delta_A$  induced by $\delta$ becomes the Eilenberg-MacLane differential

$$(\delta_A f)(v_1,v_2,\ldots, v_{n-1})$$
$$=
f(v_2,v_3,\ldots, v_{n-1})-f(v_1+v_2,v_3,v_4,\ldots,v_{n-1})+f(v_1,v_2+v_3,\ldots,v_{n-1})+\cdots$$
$$+(-1)^{n-2}f(v_1,v_2,\ldots,v_{n-2}+v_{n-1})+(-1)^{n-1}f(v_1,v_2,\ldots, v_{n-1}).$$

The map $P:F^n\to F^n$ also commutes  with translations.

\begin{lemma}
We have $PR=R {\rm Ant}$
where  ${\rm Ant}$  is the anti-symmetrization operator of $F^{n-1}$ on the space of $\Lambda^{n-1} V^*$ of  antisymmetric functions  $f(v_1,v_2,\ldots, v_{n-1})$.

\end{lemma}

\begin{proof}

To compute $P$, we have to compute

 $$v(t)=v_1+t_1(v_2-v_1)+\cdots+t_{n-1} (v_{n-1}-v_1)$$

           and

$$f(t,v_1,v_2,\ldots, v_{n-1})$$
$$=\deps \sum_{\sigma\in \Sigma[2,\ldots,n-1]}\epsilon(\sigma) f(v(t),v(t)+\epsilon_1 (v_{\sigma(2)}-v_1),\ldots, v(t)+\epsilon_{n-2}(v_{\sigma(n-1)}-v_1)).$$
Now, if $f$ is invariant by translation, we see that
 $$f(t,v_1,v_2,\ldots, v_{n-1})=\deps \sum_{\sigma\in \Sigma[2,\ldots,n-1]}\epsilon(\sigma) f(0,\epsilon_1 (v_{\sigma(2)}-v_1),\ldots,\epsilon_{n-2}(v_{\sigma(n-1)}-v_1)).$$
 We obtain the lemma.

 \end{proof}

The homotopy $G$ commutes with  translations and gives an operator $G_A$ on the complex of invariants. It follows that we obtain on the complex $\delta_A$ the relation
$$G_A\delta_A+\delta_A G_A={\rm Id}-{\rm Ant}.$$

We thus obtain that the  cohomology of the operator $\delta_A$ is isomorphic in degree $n$
to $\Lambda^{n-1}V^*$.

\section{Free variables}

Let $T_n$ be the free associative algebra in $n$ variables.
We consider $L_n\subset T_n$ as the free Lie algebra in $n$ variables.
An element $f$ of $T_n$ is written as $f(x_1,x_2,\ldots, x_n)$.

Define
$$(\delta_nf)(x_1, \ldots , x_{n+1})=\sum_{i=1}^n(-1)^{i}f(x_1, x_2, \ldots, x_{i-1},{\hat x_i}, x_{i+1},\ldots, x_n).$$

Consider $T_n(y)$ the free associative algebra  generated by $(x_1,x_2,\ldots,x_n,y)$.
An operator $h$ on $T_n$ is extended by an operator still denoted by $h$ on $T_n(y)$ where we do not operate on $y$.

We may consider the application $\tau: T_n\to T_{n}(y)$ defined by
$$(\tau_nf)(x_1, \ldots ,x_n)=f(x_1+y, x_2+y, \ldots,  x_i+ y,\ldots, x_n+y).$$

The application $\tau$ commutes with $\delta$.
Thus the kernel of $\tau$ is a subcomplex of $T_n$.
We may identify it with $T_{n-1}$ by
$(Rf)(x_1,x_2,\ldots,x_n)=f(x_2-x_1,x_3-x_2,\ldots, x_{n}-x_{n-1})$
and we obtain on $T_n$  the complex $\delta_A$ considered by Alekseev-Torossian.
Here

$$(\delta_A f)(x_1,x_2,\ldots, x_{n-1})$$
$$=
f(x_2,x_3,\ldots, x_{n-1})-f(x_1+x_2,x_3,x_4,\ldots,x_{n-1})+f(x_1,x_2+x_3,\ldots,x_{n-1})+\cdots$$
$$+(-1)^{n-2}f(x_1,x_2,\ldots,x_{n-2}+x_{n-1})+(-1)^{n-1}f(x_1,x_2,\ldots, x_{n-1}).$$

It is clear that the complex $\delta: 0\to T_0\to T_1\to T_2\cdots$ is acyclic.
Indeed we can define
$$(sf)(x_1,x_2,\ldots, x_n)=f(0,x_1,x_2,\ldots, x_n)$$
and it is immediate to verify that
$$s\delta+\delta s={\rm Id}.$$

If $f\in T_n$, we define a function $f(t,x)\in \R[t]\otimes T_k$
by the same formula as Formula (\ref{eqf}):

\begin{definition}
Let $$x(t)=x_1+t_1(x_2-x_1)+\cdots+t_{n-1} (x_{n}-x_1).$$
 Define
$$f(t,x_1,x_2,\ldots, x_{n})=\deps \sum_{\sigma\in \Sigma([2,\ldots, n])}\epsilon(\sigma) f(x(t),x(t)+\epsilon_1 (x_{\sigma(2)}-x_1),\ldots, x(t)+\epsilon_{n-1}(x_{\sigma(n)}-x_1)).$$

Define
$$(P_nf)(x_1,x_2,\ldots,x_n)=\int_{\Delta_{n-1}} f(t,x_1,x_2,\ldots, x_{n}) dt_1dt_2\cdots dt_{n-1}.$$

\end{definition}

The following lemma is immediate.
                \begin{lemma}
We  have $\delta P_n=P_n\delta.$

\end{lemma}

We define $G^1=0$ and inductively $G^{n+1}$ by the same formula as  before. More precisely, given  $f\in T^{n+1}$, we define the function $\phi$ of $T^n(x_1)$ given by
$$\phi(w_1,w_2,\ldots, w_n)=f(x_1,x_1+w_1,\ldots, x_1+w_n)$$

and define

$$(G^{n+1}f)(x_1,x_2,\ldots, x_n)=((H^n-\delta G^n)\phi)(0,x_2-x_1,\ldots, x_n-x_1).$$

  Then we conclude as before that
  $G\delta+\delta G={\rm Id}-P$.
  Restricting to the invariants,
  we obtain a map
  $G_A$ such that ${\rm Id}-{\rm Ant}=G_A\delta_A+G_A\delta_A$.
Here ${\rm Ant}$ is the anti-symmetrization operator
$\sum_{\sigma}\epsilon(\sigma) x_{\sigma(1)}\cdots x_{\sigma(n)}.$

The subspace $L_n$ of $T_n$ is stable under the differential.
The operator  ${\rm Ant}$ is equal to $0$ on $L_n$, except in degree $1,2$, as there are no totally antisymmetric elements in $L_n$ for $n\geq 3$.
Thus we obtain

\begin{theorem}

$\bullet$
The cohomology groups  $H^n(T_n,\delta_A)$ of  the complex $\delta_A:T_n\to T_n$ are  of dimension $1$ and are generated by
$\sum_{\sigma}\epsilon(\sigma)
                                x_{\sigma(1)}\cdots x_{\sigma(n)}.$

$\bullet$
The cohomology groups
$H^n(L_n,\delta_A)$
 of  the complex $\delta_A:L_n\to L_n$ are of dimension $0$ if $n>2$.
 For $n=1,2$, $$H^1(L_1,\delta_A)=\R x_1,\hspace{3cm} H^2(L_2,\delta_A)=\R [x_1,x_2].$$
\end{theorem}

Remark: The Guichardet construction also provides  an explicit homotopy.

\end{document}